\documentclass[12pt,reqno]{amsart}
\usepackage{amsmath,amsthm,amscd,amsfonts,amssymb,color}
\usepackage{cite}
\usepackage[mathscr]{eucal}
\usepackage{mathtools}

\usepackage[bookmarksnumbered,colorlinks,plainpages]{hyperref}
\newtheorem{theorem}{Theorem}[section]

\newtheorem{corollary}[theorem]{Corollary}

\theoremstyle{definition}

\newtheorem{Open Prob}[theorem]{Open Problem}
\theoremstyle{remark}
\newtheorem{remark}[theorem]{Remark}
\numberwithin{equation}{section}
\def\DJ{\leavevmode\setbox0=\hbox{D}\kern0pt\rlap{\kern.04em\raise.188\ht0\hbox{-}}D}
\begin{document}

\title[A generalization of the density zero ideal]{A generalization of the density zero ideal}

\author[S.\ Som]
{Sumit Som}

\address{           Sumit Som, Research Associate,
                    Department of Mathematics,
                    National Institute of Technology Durgapur, India.}
                   \email{somkakdwip@gmail.com}


\begin{abstract}
Let $\mathscr{F}=(F_n)$ be a sequence of nonempty finite subsets of $\omega$ such that $\lim_n |F_n|=\infty$ and define the ideal 
$$\mathcal{I}(\mathscr{F}):=\left\{A\subseteq \omega:  |A\cap F_n|/|F_n|\to 0~\mbox{as}~n\to \infty \right\}.$$
The case $F_n=\{1,\ldots,n\}$ corresponds to the classical case of density zero ideal. We show that $\mathcal{I}(\mathscr{F})$ is an analytic P-ideal  but not $F_{\sigma}$. As a consequence, we show that the set of real bounded sequences which are $\mathcal{I}(\mathscr{F})$-convergent to $0$ is not complemented in $\ell_\infty$. 
\end{abstract}

\maketitle

\section{Introduction}

Let $\mathcal{I}$ be an ideal on the nonnegative integers $\omega$, that is, a collection of subsets of $\omega$ closed under subsets and finite unions. It is also assume, unless otherwise stated, that $\mathcal{I}$ is proper (i.e., $\omega \notin \mathcal{I}$) and admissible (i.e., $\mathcal{I}$ contains that ideal $\mathrm{Fin}$ of finite sets). 
$\mathcal{I}$ is said to be a P-ideal if it is $\sigma$-directed modulo finite sets. 
Moreover, $\mathcal{I}$ is said to be a density ideal if there exists a sequence $(\mu_n)$ of finitely additive measures $\mathcal{P}(\omega) \to \mathbf{R}$ supported on disjoint finite sets such that $\mathcal{I}=\{A\subseteq \omega: \lim_n \mu_n(A)=0\}$, cf. \cite{aa}. 
Lastly, we endow $\mathcal{P}(\omega)$ with the Cantor-space-topology, hence we may speak about analytic ideals, $F_\sigma$-ideals, etc.

At this point, let $\mathscr{F}=(F_n)$ be a sequence of nonempty finite subsets of $\omega$ such that $\lim_n |F_n|=\infty$ and define the ideal 
\begin{equation}\label{eq:mainideal}
\mathcal{I}(\mathscr{F}):=\left\{A\subseteq \omega: |A\cap F_n|/|F_n|\to 0~\mbox{as}~n\to \infty\right\}.
\end{equation}
This extends the classical density zero ideal $\mathcal{Z}$, which corresponds to the sequence $(F_n)$ defined by $F_n=\{1,\ldots,n\}$ for all $n \in \omega$. Similar ideals were considered in the literature, see e.g. \cite{bb, Jarek17}.

It is easy to see 
that the function
$$
\mathsf{d}^\star_\mathscr{F}: \mathcal{P}(\omega) \to \mathbf{R}: A\mapsto \limsup_{n\to \infty} |A\cap F_n|/|F_n|.
$$
is a monotone subadditive function, cf. also \cite[Example 4]{LeoTri} and the notion of abstract upper density given in \cite{DiNasso17}. It is not difficult to show that there exists a sequence $\mathscr{F}$ such that $\mathcal{I}(\mathscr{F})\neq \mathcal{Z}$: let $F_n:=[n!,n!+n] \cap \omega$ for all $n$ and $A:=\bigcup_n F_n$. Then $A\in\mathcal{Z} \setminus \mathcal{I}(\mathscr{F})$. Our main result follows.
\section{Main Results}

\begin{theorem}\label{thm}
$\mathcal{I}(\mathscr{F})$ is a density ideal.
\end{theorem}
\begin{proof} 
It follows by \eqref{eq:mainideal} that the ideal $\mathcal{I}(\mathscr{F})$ corresponds to
$$
\textstyle \{A\subseteq \omega: \lim_{n\to \infty} \mu_n(A)=0\},
$$
where, for each $n \in \omega$, $\mu_n: \mathcal{P}(\omega) \to \mathbf{R}$ is the finitely additive probability measure defined by
$$
\forall A\subseteq \omega, \quad \mu_n(A)= |A\cap F_n|/|F_n|.
$$
This concludes the proof. 
\end{proof}

It is worth noticing that every density ideal is an analytic P-ideal, cf. \cite{aa}. It is known that every density ideal is also meager. Hence Theorem \ref{thm} implies, thanks to \cite[Corollary 1.3]{Leo2018}, the following consequence:
\begin{corollary}
The set of bounded real sequences which are $\mathcal{I}(\mathscr{F})$-convergent to $0$ is not complemented in $\ell_\infty$. 
\end{corollary} 

\begin{remark}\label{rmk}
By a classical result of Solecki, an (not necessarily proper or admissible) ideal $\mathcal{I}$ is an analytic P-ideal if and only if 
$$
\textstyle \mathcal{I}=\mathrm{Exh}(\varphi):=\{A\subseteq \omega: \lim_{n\to \infty}\varphi(A\setminus [0,n])=0\},
$$
for some lower semicontinuous submeasure $\varphi: \mathcal{P}(\omega)\to [0,\infty]$ (that is, $\varphi$ is monotone, subadditive, $\varphi(\emptyset)=0$, and $\varphi(A)=\lim_n \varphi(A\cap [0,n])$ for all $A\subseteq \omega$), cf. \cite{aa}. 
Accordingly, it is not difficult to see that, in our case, $\mathcal{I}(\mathscr{F})=\mathrm{Exh}(\varphi)$, where $\varphi$ is the lower semicontinuous submeasure defined by 
\begin{equation}\label{eq:varphi}
\textstyle 
\forall A\subseteq \omega,\quad \varphi(A):=
\sup_{n \in \omega} \mu_n(A).
\end{equation}
The proof is straightforward and left to the reader.
\end{remark}

We conclude with another property of all ideals $\mathcal{I}(\mathscr{F})$.

\begin{theorem}
$\mathcal{I}(\mathscr{F})$ is not an $F_{\sigma}$-ideal.
\end{theorem}
\begin{proof}
Let $\varphi$ be the lower semicontinuous submeasure defined in \eqref{eq:varphi}. By Remark \ref{rmk}, we have that $\mathcal{I}(\mathscr{F})=\mathrm{Exh}(\varphi)$, hence
$$
\mathcal{I}(\mathscr{F})=
\left\{A\subseteq \omega: \lim_{n\to \infty}\sup_{k \in \omega} \mu_k(A\setminus [0,n])=0 \right\}=\{A\subseteq \omega: \|A\|_\varphi=0\},
$$
where $\|A\|_\varphi:=\limsup_{n\to \infty}\mu_n(A)$. At this point, define recursively, for each $n \in \omega$, the following sets:
$$
\textstyle 
P_n:=\left(\bigcup_{k \in \omega} \left[\max F_k- \frac{|F_k|}{2^n}, \max F_k-\frac{|F_k|}{2^{n+1}}\right] \cap \omega\right)\setminus P_{n-1}.
$$
where by convention $P_{-1}:=\emptyset$. Lastly, set $G_0:=P_0\cup (\omega \setminus \bigcup_{n\in \omega} P_n)$ and $G_n:=P_n$ for all nonzero $n \in \omega$. 

By construction we have that $\{G_n: n \in \omega\}$ is a partition of $\omega$ such that: $\|G_n\|_\varphi>0$ for all $n$ and $\lim_n \|\bigcup_{k>n}G_k\|_\varphi=0$. 

The existence of such partition implies, thanks to \cite[Theorem 2.5]{MarekLeon}, that $\mathcal{I}(\mathscr{F})$ is not an $F_{\sigma}$-ideal.
\end{proof}

\medskip

\textbf{Acknowledgments.} The author is greateful to Paolo Leonetti (Bocconi University, Italy) for useful discussions.

\end{document}